\documentclass[a4paper,11pt]{article}

\usepackage{graphicx}       
\usepackage{amsfonts}
\usepackage{amsmath}
\usepackage{amsthm}

\newcommand{\N}{ \mathbb{N} }

\newcommand{\Z}{ \mathbb{Z} }

\newcommand{\R}{ \mathbb{R} }

\newcommand{\wh}[1]{ \widehat{ #1 } }
\newcommand{\wt}[1]{ \widetilde{ #1 } }

\newcommand{\calB}{\mathcal{B}}

\newcommand{\calD}{\mathcal{D}}

\newcommand{\calG}{\mathcal{G}}

\newcommand{\calS}{\mathcal{S}}

\newcommand{\eins}{\mathbf{1}}

\newcommand{\matX}{{\mathbf X}}

\newcommand{\matS}{{\mathbf S}}

\newcommand{\matV}{{\mathbf V}}

\newcommand{\vecb}{{ \mathbf b}}

\newcommand{\vece}{{ \mathbf e}}

\newcommand{\vecv}{{ \mathbf v}}

\newcommand{\vecw}{{ \mathbf w}}

\newcommand{\vecY}{{ \mathbf Y}}
\newcommand{\vecZ}{{ \mathbf Z}}

\newcommand{\bfbeta}{\beta}

\newcommand{\bfSigma}{\boldsymbol{\Sigma}}

\newcommand{\Var}{{\mbox{Var\,}}}
\newcommand{\Cov}{{\mbox{Cov\,}}}

\newtheorem{theorem}{Theorem}
\newtheorem{lemma}{Lemma}
\newtheorem{remark}{Remark}

\makeindex             

\title{On Extreme Value Asymptotics of Projected Sample Covariances in High Dimensions with Applications in Finance and Convolutional Networks}
\author{\Large Ansgar Steland \\ \ \\ Institute of Statistics and AI Center \\ RWTH Aachen University \\ Pontdriesch 14-16, 52062 Aachen \\ 
}

\date{September 2023}

\begin{document}

\maketitle

\begin{abstract} Maximum-type statistics of certain functions of the sample covariance matrix of high-dimensional vector time series are studied to statistically confirm or reject the null hypothesis that a data set has been collected under normal conditions. The approach generalizes the case of the maximal deviation of the sample autocovariances function from its assumed values. Within a linear time series framework it is shown that Gumbel-type extreme value asymptotics holds true. As applications we discuss long-only mimimal-variance portfolio optimization and subportfolio analysis with respect to idiosyncratic risks, ETF index tracking by sparse tracking portfolios, convolutional deep learners for image analysis and the analysis of array-of-sensors data. 
\end{abstract}

\section{Introduction}

An ubiquitous problem in data analysis is the problem to analyze a set of observations, in order to decide whether it satisfies a stability assumption formalizing the concept that the data set represents measurements taken under normal conditions. Let us briefly discuss two important areas where massive high-dimensional data needs to be processed and analyzed to clarify this important question: Modern industrial quality control and financial markets.

{\em Industrial Quality Control:} In modern industrial quality control normal conditions typically means that the production process is in a state of (statistical) control. Since industrial quality measurements in mass production are  taken in large numbers at high frequencies along a high-throughput production line, the resulting data are high-dimensional. Concretely, the data are often given in the form of time series data of sets of relevant numbers (such as quality features, indicators etc.), discretized measurement curves or vectorized imaging data. 
Whereas the application of the proposed method is relatively straighforward in the case of time series data, its applicability to discretized functional data and image data is somewhat less obvious but of considerable interests. Indeed, imaging technology is a well established tool in industrial quality  control to monitor the quality of produced items. Such images are frequently analyzed on-line and in real time using pretrained convolutional neural networks, in order to detect and identify either certain objects, e.g. to initiate certain actions with them, or specific defects to control production quality. We will discuss later how one can analyze the feature-generating part(s) of the network by the proposed method, in order to detect departures of the input's distribution from the assumed law under normal conditions. This means, the test is in terms of those features used and further processed by the neural net. 

{\em Financial Markets:} Similarly, financial markets produce massive amounts of data on indices, assets and derivatives. Just to give an impression: According to The World Bank there are around $ 43,000 $ exchange-listed companies worldwide, and the Financial Times has cited a study from the Index Industry Association which estimates that ca. $3$ million indices are calculated and published worldwide. Current trends are ESG and Climate Indices. ESG aims at measuring sustainability and summarizes criteria for the three key factors environmental, social and governance. There is an ongoing debate about their concrete definition and the question how to measure them for an exchange-traded company, driven by the fact that ESG ratings for an investable company may differ substatiantially across evaluating banks and firms issuing investment fonds. Another topic of critical debate is green-washing. Nevertheless, ESG investing has become an important and integral part of the investment business. Climate indices aim at measuring and tracking low carbon, greenhouse gas emissions, climate action, net zero transission, fossil fuels exclusion and EU Paris-alignment, both for equities and fixed income instruments. 

A common but doubtful assumption in practice is to define stability via the assumption that the data consist of a sample of $d$-dimensional independent and identically distributed random vectors.  But for big and high-dimensional data available today, where the dimension $d$ may be large relative to the sample size $n$,  this assumption is usually violated and needs to be replaced by the assumption that under normal conditions the observations form a strictly or weakly stationary process. Then the problem arises how to statistically confirm or reject the null hypothesis that this form of stability holds true. Whereas tests for constancy of the process mean are well studied and widely used, checking the covariance structure is more subtle and requires tools of modern high-dimensional time series analysis and stochastic process theory. To compare the (sample) covariance matrix with its population version, one can consider a global measure such as a matrix norm of their difference. However, this approach suffers from the fact that there are $ O(d^2) $ covariances and usually most of them are of minor importance or not informative at all. In high dimensions, this leads to the problem of noise accumulation when summarizing $ O(d^2) $ estimators, although the signal is low-dimensional or even sparse. This quickly results in poor performance and a substantial loss of statistical power. 

We approach the problem by focusing on selected (functions of) covariances and their estimates. Generally, when confronted with the problem to summarize a set of univariate statistics, such as differences of several estimators from their normal expectations, one needs to decide how to summarize them. If many coordinates are affected by a change, summation is a simple transformation usually leading to powerful procedures. But if only a part of the coordinates are expected to show up a change, perhaps even only a few of them, then it is more appropriate to consider the maximum. Indeed, the maximum reacts even if only one coordinate exhibits non-normal behaviour. From a statistical and probabilistic perspective, the maximum behaves quite differently from a sum. Whereas sum statistics are Gaussian in large samples under relatively mild conditions, maxima follow non-Gaussian extreme value laws and thus require a different treatment.  

In this paper, we consider the problem to analyze the (sample) covariance matrix of a high-dimensional vector time series, in order to confirm normal behaviour or detect non-normal behaviour of the correlation structure.  To be concrete, let us discuss a first explicit example and introduce notation: Let $ \vecY_t $, $ t \ge 1 $, be a mean zero $d$-dimensional vector-valued data stream with assumed covariance matrix $ \bfSigma_n $ (under normal behaviour) and denote by $ \wh{\bfSigma}_n $ the sample covariance matrix of the first $n$ data vectors $ \vecY_1, \ldots, \vecY_n $. We will assume that $ \vecY_t $ belongs to the class of  multivariate linear processes, but postpone a detailed description and assumptions to the next section. To test whether the assumed dependence model $ \bfSigma_n $ applies, we may select  $m$ entries, $ (i_1,k_1), \ldots, (i_m, k_m) $ and compare the corresponding estimates $ \wh{\Sigma}_{n,i_jk_j} $ with their hypothesized values $ \bfSigma_{n,i_j k_j} $, $ 1 \le j \le m$. A statistical test can be based on the max-deviations statistics 
\[ 
M_n = \max_{1 \le j \le m} |\wh{\bfSigma}_{n,i_jk_j} -   \bfSigma_{n,i_j k_j}|,
\] 
so that we reject the null hypothesis if this maximum is too large. There are specific problems where the number, $m$, of (sample) covariances under investigation may be small. For example, when $ \vecY_t $ are asset (or index) returns, one may be interested in analyzing certain submarkets or examining covariances of assets of such a submarket with a benchmark. The first case corresponds to a submatrix determined by the indices $ i_1, \ldots, i_l \in \{ 1, \ldots, d \}$ corresponding to the submarket's assets, i.e. one deletes all rows and columns of the (sample) covariance matrix of those assets not belonging to the submarket. The second case corresponds to all entries $(i_1,k), \ldots, (i_l,k) $, where $k$ is the index of the benchmark asset. But when the goal is a global test which examines the full covariance matrix,
 $m$ should be large. However, because usually the majority of the entries of a covariance matrix is small or even negligible, one may focus on a subset of the covariance matrix and expect that a test based on $ M_n $ is powerful even if $m$ is relatively small. Examples are structured covariance matrices or bandable covariance matrices. Structured covariance matrices arise in factor models and repeated measures designs, and typically one can assume that entries belonging to the blocks on the diagonal are much larger than entries belonging to the off-diagonal blocks. Bandable covariance matrices may arise when analyzing data from (large) sets of sensors as in environmental monitoring. If we suppose that the sensors are ordered in time or space and the $j$th coordinate of $ \vecY_t $ corresponds to the measurement from sensor $j$, it makes sense to assume that the correlation decreases with the distance between the sensors. Formally, $ \bfSigma_n $ is called bandable,
if $ | \bfSigma_{n,ik} | =  O( f(|i-k|) )$ for some strictly decreasing function $f$. Typical choices for $f$ are $ f(x) = x^{-a} $ for some $ a > 0 $. As a result, only the leading diagonals matter and the remaining part of the covariance matrix is negligible. This allows to focus on those sample covariances sitting on the main diagonals, such that again small to moderately large values of $m$ may suffice. However, theoretical results allowing that $m$ may grow with the sample size are of interest. We will come back to this point in the discussion on previous related results in the next section.

A related problem studied  in the literature to some extent is to check whether or not a stationary data stream $ Z_t $, $ t \ge 1 $, attains a prescribed (in-control) autocovariance funcion (ACF) $ \gamma(h) $, $ h \in \Z $. A natural test statistic is the maximal deviation of the sample autocovariances from their assumed population counterparts. It is well known that under mild conditions, for fixed $h$, the estimator $ \wt{\gamma}_Z(h) = \frac{1}{n} \sum_{t=1}^{n-h} Z_t Z_{t+h} $ of the lag--$h$ autocovariance $ \gamma_Z(h) $ of a time series $ Z_t$ satisfies the central limit theorem, i.e.,
$
\frac{1}{\sqrt{n}} \sum_{t=1}^{n-h} (Z_t Z_{t+h} - \gamma_Z(h)) \stackrel{d}{\to} N(0, \sigma_h^2 ),
$
as $ n \to \infty $. This holds, for example, if $ Z_t $ follows a linear process with absolutely summable coefficients, see e.g. \cite[Remark~1]{Wu2009}. Here the asymptotic variance $ \sigma_h^2 $ can be calculated using Bartlett's formula. The associated max-deviation statistic is given by
\[
V_n = \max_{0 \le h \le m} | \wh{\gamma}_Z(h) - \gamma_Z(h) |
\]
and calculates the maximal departure of the sample autocovariance from the true values up to the lag $ m $. It was conjectured by \cite{Wu2009} that $ V_n $ follows asymptotically
the Gumbel extreme value distribution when $ m_n \to \infty $. \cite{Jirak2011} verified this conjecture for linear processes and logarithmic growth of the maximal lag,
\[
m = O( \log n / \log \log n ),
\]
as $ n \to \infty $. \cite{XiaoWu2014} showed the result for nonlinear time series under the assumption $ m = O( n^{\eta} ) $ where $ 0 < \eta < 1 $. This paper complements these results by studying extreme value asymptotics for the max-type statistic $ M_n $ and a generalization thereof based on sample covariances of a linear multivariate time series.

As formulated above, the test based on $M_n$ examines selected sample covariances to evaluate whether they are in agreement with an assumed model given by $ \bfSigma_n  $. The methodology considered in this paper goes beyond this setting and allows to consider the maximal deviation of certain functions of the sample covariance matrix. The details of this extension are provided in the next section. Clearly, in order to conduct the proposed statistical test, we need the (asymptotic) distribution of $ M_n $ under the null hypothesis that $ \bfSigma_n = \Var( \vecY_t ) $ holds true. Our main theoretical result provides Gumbel extreme value asymptotics under mild conditions. 

The organisation of the paper is as follows. Section~\ref{Sec:Method} presents the proposed procedure and imposed assumptions. The asymptotics, which deals with Gumbel extreme value theory, is discussed in Section~\ref{Sec:Asymptotics}. 

\section{Maximal Deviations of Projected Sample Covariances}
\label{Sec:Method}

In the introduction we picked $m$ elements $ \wh{\bfSigma}_{n,i_jk_j} $ from the sample covariance matrix and considered the maximal departure from their theoretical values under normal conditions. Observe that these elements can be picked by calculating the bilinear form $ \vece_{i_j}' \wh{\bfSigma}_n \vece_{i_k} $ where $ \vece_\ell  = (0, \cdots, 0, 1, 0, \cdots, 0)' $ denotes the $\ell$th unit vector, $ 1 \le \ell \le m$. This observation suggests to study the generalized problem to consider the maximal deviation of $m$ bilinear forms $ \vecv_j{}' \wh{\bfSigma}_n  \vecw_j $ from the values $ \vecv_j{}'\bfSigma_n \vecw_j $, $ 1 \le j \le m $, assumed under normal conditions. Therefore, let us define
\begin{equation}
	\label{EVT}
	T_n = \max_{1 \le j \le m} | \calD_{nj} |,
	\qquad
	\calD_{nj} = \sqrt{n}  \vecv_j{}' ( \wh{\bfSigma}_n - \bfSigma_n ) \vecw_j,
	\ 1 \le j \le m.
\end{equation}
The vector time series $ \vecY_t $ is assumed to follow a multivariate linear process as in \cite{BoursSteland2021}, but we confine ourselves to the case that all coordinates $ Y_t^{(\nu)}, t \ge 1 $, $ 1 \le \nu \le d $, have coefficients with geometric decay. As well known, this is satisfied by many time series models used in practice including ARMA models. A more detailed description of this model is provided below.

The weighting vectors $ \vecv_j, \vecw_j $, $ 1 \le j \le m$, are assumed to have uniformly bounded $ \ell_1$-norms, i.e., there is some constant $ C $ such that
\begin{equation}
	\label{l1Condition}
	\sup_{j \ge 1} \| \vecv_j \|_{\ell_1} \le C, \qquad 	\sup_{j \ge 1} \| \vecw_j \|_{\ell_1} \le C.
\end{equation}
$ \ell_1$-weighing vectors naturally arise in various cases, e.g. in optimal portfolio selection and lasso regression, see the discussion in  \cite{StelandSachsBernoulli} and the financial applications discussed in detail in Section~\ref{Sec:Applications}. By a scaling trick, \cite{StelandSachsSPA}, one may also use uniformly $ \ell_2 $-bounded projections.  In Section~\ref{Sec:Applications} several applications are discussed in detail which naturally lead to projection vectors satisfying this assumption. We shall impose a further condition on the weighting vectors which is provided and discussed below.


Let us assume that $ \vecY_t $ is a multivariate linear process based on i.i.d. innovations, such that each coordinate series, $ Y_t^{(\nu)} $, $ t \ge 1 $, is a linear time series based on i.i.d. innovations with coefficients $ c_t^{(\nu)} $, $ t \ge 1 $, for $ 1 \le \nu \le d $. This means,
\[
  Y_t^{(\nu)} = \sum_{j=0}^\infty c_t^{(\nu)} \epsilon_{t-j}, \qquad t \ge 1,
\] 
for $ 1 \le \nu \le d $. We assume that $ E | \epsilon_1 |^{4+\delta} < \infty $ for some $ \delta > 0 $. The coefficients $ c_j^{(\nu)} $ are required to decay geometrically, i.e.
\begin{equation}
\label{DecayCond}
  \sup_{\nu \ge 1} | c_t^{(\nu)} | = O( \rho^{j} ), \qquad j \ge 0,
\end{equation}
for some $ 0 < \rho < 1 $. This model is well suited to describe high-dimensional series where quite often the coordinates are quite strongly correlated. But the model also allows that coordinates are independent or $r$-dependent. Recall that $r$-dependence of sequence of random variables $ Z_t $ means that $ Z_i $ and $ Z_j $ are independent if $ |i-j|> r $. For example, if for some fixed $r \in \N_0 $ it holds $ c_j^{(1)} =  0 $ for $ j \ge k$ and $ c_j^{(2)} = 0 $ if $ 0 \le j < k-r $, then $ Y_t^{(1)} $ and $ Y_t^{(2)} $ are $r$-dependent. Especially, if $ r = 0$ the series are independent, since the supports $ \calS_\nu = \{ c_t^{(\nu)} \not= 0 \} $, $ \nu = 1, 2 $, of their  coefficient sequences are disjoint. More generally, the model can host a family of $r$-dependent series. In \cite{StelandJMVA2020} it has been shown that it also includes a wide range of spiked covariance models as well as approximate VARMA models. 

We impose the following assumptions on the  weighting vectors $ \vecv_i $ and $ \vecw_i $, whose coordinates are denoted $ v_{\nu}^{(i)} $ and $ w_{\nu}^{(i)}$, respectively.

\textbf{Assumption (W1):} There exists $ 0 < \rho < 1 $ such that 
\[ \sum_{\nu=1}^d | v_{\nu}^{(i)} c_j^{(\nu)} | = O( \rho^{ji} ) \qquad \text{and} \qquad   \sum_{\nu=1}^d | w_{\nu}^{(i)} c_j^{(\nu)} | = O( \rho^{ji} ),
\]
for all $j \ge 1 $ and $i \ge 1 $.

Assumption (W1) is motivated as follows: Suppose that $ Z_t^{(\nu)} = \sum_{j=0}^\infty c_j^{(\nu)} \epsilon_{t-j} $ with coefficients $ c_j^{(\nu)} = ( \rho^{\nu}) ^j $. This means, the coordinates follow stationary autoregressive processes of order 1 with decreasing AR coefficients $ \rho^\nu $. Let $ \vecv_i = \vece_i $ and $ \vecw_k = \vece_k $. Then $ \vecv_i'\vecZ_t $ has coefficients $ \sum_{\nu=1}^d v_\nu^{(i)} c_j^{(\nu)} = \rho^{ij} $, $ j \ge 1 $, and $ \vecw_k'\vecZ_t $ has coefficients $ \sum_{\nu=1}^d w_\nu^{(k)} c_j^{(\nu)} = \rho^{kj} $, $j \ge 1 $. Therefore, Assumption (W1) means that that the coefficients of the projected series can be bounded by the decay behaviour arising when considering a family of AR model. 

Alternatively, one may require that the following assumption is fulfilled.

\textbf{Assumption (W2):} There exists $ 0 < \rho < 1 $ such that 
\[ \sum_{\nu=1}^d | v_{\nu}^{(i)} c_j^{(\nu)} | = O( \rho^{j+i} ) \qquad \text{and} \qquad   \sum_{\nu=1}^d | w_{\nu}^{(i)} c_j^{(\nu)} | = O( \rho^{j+i} ),
\]
for all $j \ge 1 $ and $i \ge 1 $.

That assumption is motivated by an array-of-sensors application discussed in detail in Section~\ref{Sec:Applications} and is also discussed in detail for subportfolios of the long-only minimial variance portfolio.

The following technical lemma shows that under Assumption (W1) resp. (W2) the covariances
\[
\beta_{ik}^{(n)} = \Cov( \vecv_i{}' \wh{\bfSigma}_n \vecw_i, \vecv_k{}' \wh{\bfSigma}_n \vecw_k  )
\]
decay geometrically and hence the correlations as well. 	In \cite{StelandSachsBernoulli} it has been shown that
\[
\beta_{ik}^{(n)} = F_{ni} F_{nk} ( \gamma - \sigma^4 ) + F_{n,ik}\sigma^4
\]
where $ \sigma^2 = \Var( \epsilon_1 ) $, $ \gamma = E( \epsilon_1^3 ) $ and
\begin{align*}
	F_{ni} & = 2 \sum_{\nu, \mu=1}^{d} v_\nu^{(i)} w_\mu^{(i)} \sum_{j=1}^{\infty}
	c_j^{(\nu)} c_j^{(\mu)}, \\
	F_{n,ik} & = \sum_{l=1}^\infty \sum_{\nu, \mu=1}^{d} \sum_{\nu', \mu' = 1}^{d}
	v_\nu^{(i)} w_\mu^{(i)} v_{\nu'}^{(k)} w_{\mu'}^{(k)} \sum_{j=1}^{\infty} c_{jl}(\nu, \mu, \nu', \mu' ),
\end{align*}
with 
\[ c_{jl}( \nu, \mu, \nu', \mu' ) = [c_j^{(\nu)} c_{j+l}^{(\mu)} + c_j^{(\mu)} c_{j+l}^{(\nu)}][c_j^{(\nu')} c_{j+l}^{(\mu')} + c_j^{(\mu')} c_{j+l}^{(\nu')}]. \] 
Here the entries of the weighting vectors are assumed to be fixed, i.e. they do not depend on $n$. Then the limits $ \beta_{ik} $ are obtained by formally replacing $d$ by $ \infty $ in the above formulas. The case that the entries depend on $n$ is examined below for a large class of weighting vectors. 

\begin{lemma}
	\label{lemma1} Suppose that Assumption (W1) or (W2) is fulfilled. Then $ \beta_{ik}^{(n)}  \to \beta_{ik} $, $ n \to \infty $, for all $ i, k \ge 1 $. Further,
	\[ \beta_{ik}^{(n)} = O( \rho^h ), \qquad \text{if $ |i-k| \ge h $},\]
	and the same applies to their limits.
\end{lemma}

\begin{remark} The assertion of Lemma~\ref{lemma1} can be used as a top-level assumption replacing the more concrete Assumptions (W1) and (W2), respectively.
\end{remark}

As an idealized setting for the case that the weighting vectors are estimators for true weighting vectors $ \vecv_i, \vecw_i $ satisfying our assumptions, let us assume that they depend on the sample sizes and ensure that the relative error is of the order $ O( r_n ) $ for some sequence $ r_n \to 0 $. This means, we assume that one uses projection vectors $ \vecv_{ni} = ( v_{n\nu}^{(i)} )_{\nu=1}^{d_n} $ and $ \vecw_{ni} = ( w_{n\nu}^{(i)} )_{\nu=1}^{d_n} $ satisfying,
\begin{equation}
\label{AssWeights}
	 \sup_{\nu, i \ge 1} \frac{| v_{n\nu}^{(i)} - v_\nu^{(i)} |}{ |v_\nu^{(i)}| } = O(r_n),
	 \quad 
	 \sup_{\nu, i \ge 1} \frac{| w_{n\nu}^{(i)} - w_\nu^{(i)} |}{ |w_\nu^{(i)} | } = O(r_n).
\end{equation}
Let us denote by $ \hat{\beta}_{ik}^{(n)} $ the associated covariances obtained when using the projections $ \vecv_{ni} $ and $ \vecw_{nk} $.

\begin{lemma}
	\label{lemma2} Suppose that the weighting vectors  $ \vecv_{ni} $, $ \vecw_{nk} $ satisfy (\ref{AssWeights}) holds. Then 
	\[
	  \hat\beta_{ik}^{(n)} = \beta_{ik} + O(r_n), \qquad \hat\beta_{ik} = O( \rho^h ), \quad |i-k| \ge h. 
	\]
\end{lemma}

\section{Gumbel Extreme Value Asymptotics}
\label{Sec:Asymptotics}

There are various approaches to establish an extreme value limit. As well known, there are only three possible limiting distributions for the (normalised) maximum $ M_n^* = \max( X_1, \ldots, X_n) $ of $n$ i.i.d. random variables $ X_1, \ldots, X_n $: the Weibull law, the Gumbel distribution and the Fr\'echet law. This can be determined by studying the tail behaviour of the distribution of the $ X_i$'s. In particular, the normal distribution belongs to the attractor of the Gumbel law. In view of the asymptotic normality of the sample covariances and sample autocovariances of a stationary time series, this is the candidate extreme value distribution in this particular case. And since the Gaussian approximations of general bilinear forms of a sample covariance matrix for high-dimensional (linear) time series obtained in \cite{StelandSachsBernoulli}, \cite{StelandSachsSPA}, \cite{BoursSteland2021}, \cite{StelandJMVA2020} and \cite{MiesSteland2022} also imply asymptotic normality, the candidate extreme value asymptotics for $T_n $ is the Gumbel law as well. 

But here we have to deal with the maximum of dependent random variables, even if we take the Gaussian approximation step for granted, and when the $X_i$ are correlated, the situation is more involved even under Gaussianity. Indeed, then a Gumbel extreme value asymptotics may not apply due to clustering of extremes. But if the correlations die out quickly enough in the sense of Berman's condition, the Gumbel extreme value law still applies.
As discussed in \cite{Leadbetter1974}, Berman's condition also ensures the validity of the asymptotic mixing condition used for non-Gaussian series, which requires that the distribution function of finitely many coordinates has asymptotically product form up to an error $ g(k) = o(1)$, if there is a gap of length $k$ in the coordinates. 

In view of our Lemma~\ref{lemma1}, we can draw on a Berman-type condition to establish Gumbel asymptotics. However, some care is needed and there are further complications. Firstly, because even after approximating our sequence $ \calD_{nj} $ by a Gaussian counterpart, we are still left with a {\em non-stationary} Gaussian sequence and need to rely on a strenghtened Berman condition ensuring that  extreme value theory applies for a non-stationary Gaussian sequence. Second, we need to  approximate $ \calD_{nj} $, $ 1 \le j \le m$, by a Gaussian process. Here we rely on recent works which provide such approximations for high-dimensional time series (i.e. where the dimension may tend to $ \infty $). Taking into account the convergence rate imposes a  constraint on the growth of the number $m$ of bilinear forms we can take. 

The following main result now provides the extreme value asymptotics for $ T_n = \max_{1 \le j \le m} | \calD_{nj} | $ and shows that $ T_n $ is in the extreme value attractor of the Gumbel law. We also show a related result for the maximum of the differences, $ \max_{1 \le j \le m} \calD_{nj} $. In this case a different centering sequence is required.

\begin{theorem} 
	\label{EVT}
	Suppose that $ \vecY_t $ is a multivariate linear process with i.i.d. innovations, such that each coordinate process is a linear process with i.i.d. innovations posessing a finite absolute moment of order $ 4 + \delta $ for some $ \delta > 0 $ and geometrically decaying coefficients. If Assumption (W1) or (W2) holds, then there exists some $  0 < \lambda < 1/6 $ such that if
	\begin{equation}
		\label{EVTConditionLn}
		m \sqrt{2 \log m } = o( n^\lambda ),
	\end{equation}
	as $ n \to \infty $, then we have for $ z \in \R $
	\[
	P\biggl( a_m^{-1} \bigl( \max_{1 \le j \le m} | \calD_{nj} |/\beta_{jj}^{1/2} - b_m \bigr) \le z \biggr) \to \exp( e^{-z} ), 
	\]
	as $ n \to \infty $, where $ a_m =   (2 \log m )^{-1/2} $ and
	\[ b_m  = \{ (2 \log m )^{1/2} - (8 \log m)^{-1/2}( \log \log m + 4 \pi - 4) \}. \]
	Further,
	\[
	P\biggl( a_m^{-1} \bigl( \max_{1 \le j \le m}  \calD_{nj} /\beta_{jj}^{1/2} - c_m \bigr) \le z \biggr) \to \exp( e^{-z} ), 
\]
as $ n \to \infty $, where 
\[ 
	c_m =   (2 \log m )^{-1/2}- \frac{1}{2}(2 \log m)^{-1/2}( \log \log m + 4 \pi).
\]
\end{theorem}

The above result suggest the asymptotic simultaneous confidence intervals
\[
\vecv_j' \wh{\bfSigma}_n \vecw_j \pm \frac{\wh{\beta}_{jj} q_m}{ \sqrt{m} }, \qquad j = 0, \dots, m,
\]
where $ \wh{\beta}_{jj} $ is a consistent estimator of the asymptotic standard deviation $ \beta_{jj} $ and
\[
	q_m = a_m G^{-1}( 1-\alpha ) + b_m,
\]
where $ G^{-1}(x) = - \log( - \log( 1 - \alpha ) ) $, $ x \in (0,1) $, denotes the quantile function of the Gumbel distribution.
In order to estimate the asymptotic variances, one may use the Bartlett-type estimators studied in \cite{StelandSachsSPA}, \cite{StelandJMVA2020} and \cite{BoursSteland2021}.

\section{Applications}
\label{Sec:Applications}

When investing in financial markets with $d$ assets one needs to set up a portfolio $ \vecw \in \R^d $ whose time $t$ return equals $ \vecw' \vecY_t $, if $ \vecY_t $ is the vector of (daily, weekly or monthly) asset returns. Thus, the projections analyzed by the proposed approach correspond to portfolio returns, their variances represent the risk associated with the portfolio and the covariances measure the dependence between portfolios. Below we discuss two concrete applications, namely the construction of long-only portfolios minimizing the portfolio risk and index tracking to physically replicate an index.

Projections also arise naturally in convolutional neural networks and, indeed, in large numbers. These networks are a common tool in image analysis, object detection and classification to mention only a few key areas. Here projections $ \vecv' \vecY_t $ serve as first principle features calculated from the (raw) input data $ \vecY_t$,  which are fed into later stages of the neural net. Internally, the neural net computes derived features and combinations thereof (wich can mimic logic combinations), and finally linearly combines them to produce an output. 

Lastly, we consider data as arising in an idealized environmental application where a large number of sensors is used to receive a signal, but also in quality control where sensors are placed at different locations to measure vibrations and resonance of materials such as airplane wings. Due to different distances to the signal source the sensors receive a more or less damped signal.

\subsection{The Long-Only Minimum-Variance Portfolio in One Factor Models}

Suppose that $ \vecY_t $ are $d$ mean zero (excess) asset returns with positive definite covariance matrix $ \bfSigma $. Consider the global minimum-variance portfolio (MVP) $ \vecw^* \in \R^d $ defined by the minimization problem
\[ 
\min_{\vecw \in \R^d} \vecw' \bfSigma \vecw, \quad \text{such that} \quad \vecw'\eins = 1,
\]
The constraint is a budget constraint which allows for long as well as short position. It is well known that there is a closed-form formula for the solution, namely
\[
	\vecw^* = \frac{ \bfSigma^{-1} \eins }{\eins'\bfSigma^{-1} \eins}.
\]
Moreover, any rational investor in the Markowitz sense holds a linear combination of the market
If short sales are not allowed, we are interested in the long-only minimum-variance portfolio (LMVP) 
\[
\min \vecw' \bfSigma \vecw, \quad \text{such that\ \ } w_i \ge 0, 1 \le i \le d, \ \vecw'\eins = 1,
\]
which adds the long-only constraints. In general, there are no explicit solutions so that one has to rely on numerical algorithms. However, in the single factor model, a simple closed-form solution for the GMVP can be derived which carries over to a semi-closed form solution for the LMVP, see \cite{ClarkeSilvaThorley2011} and \cite{Qi2021}. The single factor model assumes that 
\[
Y_t^{(\nu)} = \beta_\nu (R_{Mt}-r) + \epsilon_t^{(\nu)}
\]
where $ R_{Mt} $ stands for the time $t$ return of the factor (e.g. the market represented in practice by an index such as the S\&P~500 or a capitalization-weighted average of the investable universe) with variance $ \sigma_M^2 $, $r$ is the risk-free rate of return, $ \beta_1, \ldots, \beta_d $ are the beta factors and $ \epsilon_t^{(\nu)} $ the mean zero idiosyncratic errors with variances $ \sigma_\nu^2 $, uncorrelated across $ \nu $. In the following it is assumed that all beta factors are nonnegative. Put $ \vecb = (\beta_1, \ldots, \beta_d)' $. Then 
\[ \bfSigma = \sigma_M^2 \vecb \vecb' + \matS, \qquad \matS = \text{diag}(\sigma_1^2, \ldots, \sigma_d^2) \]
and hence 
\[ \bfSigma^{-1} = \matS^{-1} - \frac{\vecb_r \vecb_r'}{\sigma_M^{-2} + \vecb_r' \vecb}, 
\] 
where $ \vecb_r = (\beta_1/\sigma_1^2, \ldots, \beta_d / \sigma_d^2)' $ is the vector of risk-adjusted beta's and $ \matS^{-1} = \text{diag}(\sigma_1^{-2}, \ldots, \sigma_d^{-2}) $. One gets the explicit closed-form optimal solution for the MVP weights
\begin{equation}
\label{MVP-Weights}
w_\nu^* = \frac{\sigma_{MVP}^2}{\sigma_\nu^2}\left(1- \frac{\min(\beta_i,\beta_{LS})}{\beta_{LS}} \right), \qquad 1 \le \nu \le d.
\end{equation}
Here $ \beta_{LS} = \frac{\sigma_{MVP}^{-2} + \sum_{i=1}^d \beta_i / \sigma_i^2 }{ \sum_{i=1}^d \beta_i / \sigma_i^2 }$ is the long-short threshold beta.  One buys only assets with beta's smaller than this threshold. Assets with beta factors exceeding $ \beta_{LS} $ are shorted. It has been conjectured in \cite{ClarkeSilvaThorley2011} and shown in \cite{Qi2021} that the long-only minimum-variance portfolio attains the same formula with $ \beta_{LS} $ replaced by the long-only threshold beta given by the smallest solution of the equation
\[ 
\beta_{LO} = \frac{\sigma_{MVP}^{-2} + \sum_{\beta_i < \beta_{LO}} \beta_i / \sigma_i^2 }{ \sum_{\beta_i < \beta_{LO}} \beta_i / \sigma_i^2 }.
\]
Therefore, the  LMVP weights are given
\[
	w_{L,\nu^*} = \frac{1}{\sigma_\nu^2} \left( 1 - \frac{\min(\beta_\nu,\beta_{LO})}{\beta_{LO}}  \right).
\]
Denote the resulting portfolio by $ \vecw_{LO}^* = (w_{LO,1}, \ldots, w_{LO,d})' $. 
The optimal portfolio has only (long) positions in all investable assets with beta factors not exceeding $\beta_{LO} $, i.e. in low-beta assets. Observe that due to the budget constraint one may ignore the factor $ \sigma_{MVP}^2 $ in the formula for $ w_\nu^* $: One can calculate the weights assuming $ \sigma_{MVP}^2 = 1 $ and then divide these values by their sum to ensure the budget constraint.  

Since the porfolio is long-only by construction, the budget constraint also ensures that 
\[ 
	\| \vecw^*_{LO} \|_1 = \sum_{\nu=1}^d |w_{LO,\nu} | = \sum_{\nu=1}^d w_{LO,\nu} = 1.
\] 
Consequently, the LMVP  has uniformly bounded $ \ell_1 $-norm.

The basic idea of the following approach to analyze the market in terms of the LMVP is to form subportfolios of assets with similar idiosyncratic risks. We may assume that the assets are ordered with respect to their idiosyncratic risks, 
\[
  \sigma_1^2 \le \cdots \le \sigma_d^2.
\]
Group now the assets in $ g $ groups of  assets, such that the $k$th group consists of assets $ \calG_k = \{ i_k, \ldots, i_{k+1}-1 \} $ with idiosyncratic risks $ \sigma_{i_k}^2, \ldots, \sigma_{i_{k+1}-1}^2 $. This means, the assets are grouped according to their idiosyncratic risk in $g$ risk classes. Define the weighting vector $ \vecv_i $ by $ \vecv^{(i)}_{\nu} =   \vecw^*_{LO,\nu} $ if $ \nu \in \calG_i  $ and $ \vecv^{(i)}_{\nu}  = 0 $ if $ \nu \not\in \calG_i $, $ 1 \le i \le g $. Observe that $ \vecv_i $ corresponds to the projection onto the long-only sub-portfolio corresponding to the assets of the risk class $ \calG_i $ and $ \vecv_1 + \ldots + \vecv_g = \vecw_{LO}^* $. In other words, the subportfolios $ \vecv_i $ decompose the long-only optimal portfolio with respect to their idiosyncratic risks.

Let us assume that these idiosyncratic risks of risk class $i$  satisfy the growth condition
\begin{equation}
\label{OrderRestriction}
  \sigma_\nu^2 \ge C \nu^{1+\delta} \rho^{-i}, \qquad \nu \in \calG_i , 
\end{equation}
for some $ \delta > 0 $, where $ 0 < \rho < 1 $ is from the decay condition (\ref{DecayCond}) for the coefficients of the underlying linear processes. Let us briefly argue that the above condition can hold true. For that purpose assume equality in the condition. If we denote $ \nu_{\min}^{(i)} = \min \calG_i $ and $ \nu_{\max}^{(i)} = \max \calG_i $, then the order relation of variances holds, if \[ ( \nu_{\max}^{(i)} )^{1+\delta} \rho^{-i} < (\nu_{\min}^{(i+1)} )^{1+\delta} \rho^{-(i+1)} \] or equivalently 
\[
  \left( \frac{\nu_{\max}^{(i)} }{ \nu_{\min}^{(i+1)} } \right)^{1+\delta} < \frac{1}{\rho}.
\]
for all $i$. Since $ \nu_{\max}^{(i)} /\nu_{\min}^{(i+1)} = \nu_{\max}^{(i)} /(\nu_{\max}^{(i)} +1) < 1 $ and $ 1/\rho > 1 $, these inequalities are fulfilled for small enough $ \delta > 0 $.

Under assumption (\ref{OrderRestriction}), the (L)MVP weights corresponding to group $ \calG_i $, i.e. the non-vanishing entries $ v_{\nu}^{(i)} $, $ \nu \in \calG_i $, of $ \vecv_i $, satisfy
\[
  | v_{\nu}^{(i)} | =   w_{LO,\nu}^*  \le C \sigma_{MVP}^2 \nu^{-(1+\delta)} \rho^{i}.
\]
Therefore, we obtain
\[
  \left| \sum_{\nu=1}^d v^{(i)}_\nu c_j^{(\nu)} \right| = O\left( \sum_{\nu=1}^d \nu^{-1-\delta} \rho^{i+j} \right) = O( \rho^{i+j} ),
  \]
for all $ 1 \le i \le g $. It follows that the (L)MVP subportfolios weighting vectors satisfy Assumption (W2), so that the established results allow to compare these portfolios with all portfolio vectors which fulfill (W2) as well.

\subsection{Lasso-Regression for Index Tracking}

Financial indices aim at representing markets and therefore may consist of a large number of investable assets. For example, the MSCI All Countries World Index (MSCI ACWI) captures large, mid and small cap investable stocks of 23 developed markets and 24 emerging markets countries. It consists of 9,123 constituents and covers approximately 99\% of the global equity investment opportunity set. There are many exchange traded funds (ETFs) which physically replicate the index, in order to allow investors to invest into the world market measured by the MSCI AWCI. Clearly, holding the right positions in all 9,123 stocks and update them on a regular basis would result in high trading costs and could be even challenging or impossible, since not all companies included in the index are liquid. Therefore, issuers of such ETFs try to track the underlying index as closely as possible by investing only in a subset of the universe underlying the index. Thus, the problem is to replicate as closely as possible the index price resp. return series by some $s$-sparse portfolio vector. This problem is called index tracking, and the tracking error is an important quantity published by the issuers. 

Here, one may use lasso regressions to track the index, say $Z_t$, see \cite{YangWu2016}. The corresponding regression model is  $ Z_t = \bfbeta^\top \vecY_t + e_t $,  where
	$ \vecY_t $ serve as a $d$-dimensional regressor (consisting of the underlying investable universe) and $ e_t $ are mean zero error terms (representing the tracking error). $ \bfbeta$ is assumed to be an $ s $-sparse portfolio vector. The lasso approach yields an 
	 $ \ell_0 $-sparse estimator $ \wh{\bfbeta}_n $, where the number, say $s$, of non-vanishing entries of the estimator can be controlled by the choice of the penalty parameter to balance the number of active positions and the tracking error.
 The lasso provides us with a $s$-sparse portfolio $ \wh{\bfbeta}_n $ with $s$ active positions such that $ \wh{\bfbeta}_n^\top \vecY_t $ approximates the index  $ Z_t $. Decomposing $ \wh{\beta}_n $ as in the previous application into $g$ groups according to the ordering of the absolute values, $ | \wh{\beta}_{n,\nu} | $, one may show that Assumption (W2) holds under similar conditions.

\subsection{Convolutional Neural Networks in Image Analysis}

For a convolutional neural network designed to classify image data the question arises whether the network classifier is still valid and applicable when the data distribution changes. The statistic $ M_n $ can be used to analyze this question by applying it to the convolutional projections computed by the net at the first stage.
 Suppose that  $ \matV_i $, $ 1 \le i \le m $, are $m$ convolution matrix operators of a  convolutional layer of a pretrained deep neural network, \cite{LecunEtAl1998}. It is common practice to train the network in such a way that all convolutional filters $ \matV_i $, $ 1 \le i \le m $, are  $s$-sparse for some integer $ s $. For simplicity of presentation, we assume that the convolutional layer is the first hidden layer and thus  processes an input image $ \matX $ represented by, say, a $ n_x \times n_y $, matrix with entries  $ X_{k\ell} $. The convolutional layer  computes the inner products $ Z_i = \sum_{k=1}^{n_x} \sum_{\ell=1}^{n_y} V_{i,k \ell} X_{k\ell} $, $ 1 \le i \le m $, which can be written in the form $ Z_i = \vecv_i' \vecY $, if $ \vecY = \text{vech}(\matX) $ and $ \vecv_i = \text{vech}( \matV_i ) $, where $ \text{vech} $ denotes the vectorization operator stacking all columns (or rows) of a matrix. We see that the projection vectors $ \vecv_i $ arising in a convolutional layer of a deep learner are $s$-sparse. If we assume that these vectors defining the convolutional layer can be order in such a way that 
 \[
 | v_{\nu}^{(i)} | \le  C \nu^{-(1+\delta)} \rho^{i}.
 \]
 holds for some constant $ C $ and some $ \rho \in (0,1) $, then Assumption (W2) follows as in the first application. Indeed, for many problems the task is essentially to identify objects and related details in the center of the images. So the corresponding convolutional filters will concentrated on the respective rows. Thus, we can order the rows according to a measure of the filter coefficients such as their $ \ell_1 $-norm, so that filters identifying details should get smaller measures than more global filters. Hence, after appropriately permuting it is not restrictive to assume that the vectors $ \vecv_i $ obtained by stacking rows satisfy the above condition.

If the vectorized image $ \vecY $ corresponds to a produced item under normal condition, we may assume that its covariance matrix is given by the in-control covariance matrix under normal conditions. This matrix can  be assumed to be known, as it can be estimated from defect-free produced items. But if the distribution of $ \vecY $ changes, that change can be detected by the statistic $ M_n $ which examines the maximal deviation of the projected sample covariance from its in-control projection, where the projections are the (trained) features used by the neural network. This means, we analyze a possible departure from the in-control state in terms of the convolutional layer of the given neural network. If a departure is detected, this indicates a  change of the underlying distribution. Since this could result in suboptimal convolutional filters, one may trigger a retraining of the neural network. 

\subsection{Array of Sensors Model}

Suppose that $ \vecY_t $ corresponds to measurements from $d$ sensors, so that $ Y_t^{(\nu)} $ represents the reading from sensor $ \nu $. Let us assume that the sensors receive a signal from a fixed location and that the sensors are located such that the signal strength decreases from sensor $ \nu $ to sensor $ \nu+1 $ by a deterministic damping factor  $D_{\nu+1}$,
\[
Y_t^{(\nu+1)} = D_{\nu+1} Y_t^{(\nu)}, \qquad \nu \ge 1.
\]
For simplicity, let us assume that the $i$th sensors is located on a circle around the signal source with radius $ r_i$ where $ r_1 < \cdots < r_d $. Let us call sensor $ i+ \ell$ the $ \ell$th neighbor of sensor $i$. Assume that sensor $ \nu = 1 $ records the filtered series, $ Y_t^{(1)} = \sum_{j=1}^\infty c_j^{(1)} \epsilon_{t-j} $, for some strong white noise input process $ \{ \epsilon_t \} $, according to our model, with coefficients $ c_j^{(1)} \le  C \rho^j $, $ j \ge 0 $, for some $ 0 \le \rho < 1 $ and a constant $ C$. If we assume that the damping factors satisfies 
\[ 
D_\nu \le \rho, \qquad \text{for all $ 2 \le \nu \le d $},
\] 
then the coefficients satisfy
\[ 
c_j^{(\nu)} = D _\nu c_j^{(\nu-1)} = \cdots = \prod_{k=2}^\nu D_k  c_j^{(1)} \le C  \rho^{\nu-1+j} = K \rho^{\nu+j} = O( \rho^{\nu+j}), \quad j \ge 0,
\] 
with $ K = C/\rho $. Thus, Assumption (W2) is satisfied if $ \vecv_i = \vecw_i = \vece_i $ corresponding to an analysis of all variances. Similarly, for the choice $ \vecv_i = \vece_i $ and $ \vecw_i = \vece_{i+\ell} $ for some fixed $ \ell \in \N $, which corresponds to an analysis of the covariances of all sensors with respect to their $\ell$th neigbor, 
since $ \sum_{\nu=1}^d w_\nu^{(i)} c_j^{(\nu)} = c_j^{(i+\ell)} = O( \rho^{i+\ell+j}) = O( \rho^{i+j}) $. 

\section{Proofs}

\begin{proof}[Proof of Lemma~\ref{lemma1}] Let us first assume that Assumption (W1) holds. 
	We provide explict derivations and start with $ F_{ni} $.  We have
	\[
	\left( \sum_{\nu=1}^d | v_\nu^{(i)} | | c_j^{(\nu)} |  \right) \left( \sum_{\mu=1}^d | w_\mu^{(i)} | | c_j^{(\mu)} | \right)   = O( \rho^{2ji} ), \qquad j \ge 0, i \ge 1,
	\]
	and therefore, after rearranging terms, 
	\begin{align*}
		| F_{ni} | & \le \sum_{j = 1}^\infty \left( \sum_{\nu=1}^d | v_\nu^{(i)} | | c_j^{(\nu)} |  \right) \left( \sum_{\mu=1}^d | w_\mu^{(i)} | | c_j^{(\mu)} | \right)  \\
		&\le C  \rho^i \sum_{j=1}^\infty \rho^{2j} \\
		& = \frac{C}{1-\rho^2} \rho^i, 
	\end{align*}
	for some constant $ C $. Using the bound $ \sup_{i,k \ge 1} \rho^{2i} = \rho^2 \le 1 $ we obtain
	for $ k \ge i + h $, so that $  \rho^{i+k} \le \rho^{2i+h} \le \rho^h$ for all such pairs $i,k $,
	\[
	\sup_{|i-k|\ge h} | F_{ni} F_{nk} | = O( \rho^{k+i} ) = O( \rho^h ).
	\]
	The case $ i \ge k +h $ is treated analogously.	
	
	To estimate the term $ F_{n,ik} $ multiply out the expression $  c_{jl}( \nu, \mu, \nu', \mu' )  $ and rearrange expressions to represent $ F_{n,ik} $ as a sum of four terms, 
	\[ F_{n,ik} = F_{n,ik}^{(1)} + \cdots + F_{n,ik}^{(4)}, \] 
	which can be estimated separately. The first term corresponding to $ c_j^{(\nu)} c_{j+l}^{(\mu)} c_j^{(\nu')} c_{j+l}^{(\mu')}  $ is given by
	\[
	F_{n,ik}^{(1)} = \sum_{\ell=1}^\infty \sum_{j=0}^\infty
	\left( \sum_{\nu=1}^d v_{\nu}^{(i)} c_j^{(\nu)} \right)
	\left( \sum_{\mu=1}^d w_\mu^{(i)} c_{j+\ell}^{(\mu)} \right)
	\left( \sum_{\nu'=1}^d v_{\nu'}^{(k)} c_{j}^{(\nu')} \right)
	\left( \sum_{\mu'=1}^d w_{\mu'}^{(k)} c_{j+\ell}^{(\mu')} \right).
	\]
	It can be bounded by 
	\begin{align*}
		| F_{n,ik}^{(1)} | &\le C_1 \sum_{\ell=1}^\infty \sum_{j=1}^\infty \rho^{2ij+2kj +\ell(i+k)}
	\end{align*}
	for some constant $ C_1 $. Since $ \rho^{(i+k)\ell} \le \rho^\ell $, $ \rho^{(i+k)j} \le \rho^j $ and $ \rho^{(i+k)j} \le \rho^{i+k} $, we have
	\[
	\rho^{2ij+2kj +\ell(i+k)} = \rho^{(i+k)\ell} \rho^{(i+k)j} \rho^{(i+k)j} 
	\le \rho^\ell \rho^{i+k} \rho^j.
	\]
	For all $i,k \ge 1$ with $ k \ge i + h $ it holds $ \rho^{i+k} \le \rho^{2i} \rho^{h} \le \rho^h $, and if $ i \ge k + h$, then $ \rho^{i+k} \le \rho^{2k} \rho^h \le \rho^h $. Hence, for all $i,k $ with $ |i-k| \ge h$ we obtain
	\[
	| F_{n,ik}^{(1)} | \le  C_1 \sum_{\ell=1}^\infty \rho^\ell \sum_{j=1}^\infty \rho^j \rho^{h} = O( \rho^h ).
	\]
	
	Next consider the second term, $ F_{n,ik}^{(2)} $, corresponding to 
	$ c_j^{(\nu)} c_{j+\ell}^{(\mu)} c_j^{(\mu')} c_{j+\ell}^{(\nu')} $, i.e.
	\[
	F_{n,ik}^{(2)} = \sum_{\ell=1}^\infty \sum_{j=0}^\infty
	\left( \sum_{\nu=1}^d v_{\nu}^{(i)} c_j^{(\nu)} \right)
	\left( \sum_{\mu=1}^d w_\mu^{(i)} c_{j+\ell}^{(\mu)} \right)
	\left( \sum_{\nu'=1}^d v_{\nu'}^{(k)} c_{j+\ell}^{(\nu')} \right)
	\left( \sum_{\mu'=1}^d w_{\mu'}^{(k)} c_{j}^{(\mu')}  \right).
	\]
	It can be bounded by 
	\[
	| F_{n,ik}^{(2)} | = O( \rho^{ji+(j+\ell)i+(j+\ell)k+jk} ) = O( \rho^{\ell(i+k)} \rho^{2j(i+k)})
	\]
	Hence, using the same arguments as above, $ | F_{n,ik}^{(2)} |  = O( \rho^h ) $ follows for all $i,k $ with $ |i-k| \ge h $.
	
	Now consider the term, $ F_{n,ik}^{(3)} $, corresponding to $ c_j^{(\mu)} c_{j+\ell}^{(\nu)} c_j^{(\nu')} c_{j+\ell}^{(\mu')} $, i.e.
	\[
	F_{n,ik}^{(3)} = \sum_{\ell=1}^\infty \sum_{j=0}^\infty
	\left( \sum_{\nu=1}^d v_{\nu}^{(i)} c_j^{(\mu)}\right)
	\left( \sum_{\mu=1}^d w_\mu^{(i)} c_{j+\ell}^{(\nu)}\right)
	\left( \sum_{\nu'=1}^d v_{\nu'}^{(k)} c_j^{(\nu')} \right)
	\left( \sum_{\mu'=1}^d w_{\mu'}^{(k)} c_{j+\ell}^{(\mu')}  \right).
	\] 
	which can be bounded by $ O( \rho^{ji+jk+(j+\ell)(i+k)}) = O( \rho^{(i+k)\ell} \rho^{2j(i+k)} )$ already treated above.
	Lastly, consider $ F_{n,ik}^{(4)} $ corresponding to $ c_j^{(\mu)} c_{j+\ell}^{(\nu)} c_j^{(\mu')} c_{j+\ell}^{(\nu')} $, i.e.
	\[
	F_{n,ik}^{(4)} = \sum_{\ell=1}^\infty \sum_{j=0}^\infty
	\left( \sum_{\nu=1}^d v_{\nu}^{(i)} c_j^{(\mu)}\right)
	\left( \sum_{\mu=1}^d w_\mu^{(i)} c_{j+\ell}^{(\nu)}\right)
	\left( \sum_{\nu'=1}^d v_{\nu'}^{(k)}  c_{j+\ell}^{(\nu')}   \right)
	\left( \sum_{\mu'=1}^d w_{\mu'}^{(k)} c_j^{(\mu')} \right).
	\]
	This term can be bounded by $ O( \rho^{ij+(j+\ell)(i+k)+jk}) = O( \rho^{(i+k)\ell} \rho^{2j(i+k)} )$ as well. 
	
	Now let us assume Assumption (W2). Then
	\[
	\left( \sum_{\nu=1}^d | v_\nu^{(i)} | | c_j^{(\nu)} |  \right) \left( \sum_{\mu=1}^d | w_\mu^{(i)} | | c_j^{(\mu)} | \right)   = O( \rho^{2(j+i)} ), \qquad j \ge 0, i \ge 1,
	\]
	and therefore, after rearranging terms, 
	\begin{align*}
		| F_{ni} | & \le \sum_{j = 1}^\infty \left( \sum_{\nu=1}^d | v_\nu^{(i)} | | c_j^{(\nu)} |  \right) \left( \sum_{\mu=1}^d | w_\mu^{(i)} | | c_j^{(\mu)} | \right)  \\
		&\le C  \rho^{2i} \sum_{j=1}^\infty \rho^{2(j+i)} \\
		& = \frac{C}{1-\rho^2} \rho^i, 
	\end{align*}
	Therefore, $ | F_{ni} | | F_{nk} | = O( \rho^{i+k} ) = O( \rho^h ) $, if $ |k-i|\ge h $. Next consider $ F_{n,ik}^{(1)} $. 
	\begin{align*}
		F_{n,ik}^{(1)} & = \sum_{\ell=1}^\infty \sum_{j=0}^\infty
		\left( \sum_{\nu=1}^d v_{\nu}^{(i)} c_j^{(\nu)} \right)
		\left( \sum_{\mu=1}^d w_\mu^{(i)} c_{j+\ell}^{(\mu)} \right)
		\left( \sum_{\nu'=1}^d v_{\nu'}^{(k)} c_{j}^{(\nu')} \right)
		\left( \sum_{\mu'=1}^d w_{\mu'}^{(k)} c_{j+\ell}^{(\mu')} \right) \\
		& = O \left( \sum_{\ell=1}^\infty \sum_{j=0}^\infty \rho^{i+j+i+j+\ell+k+j+k+j+\ell} \right) \\
		& = O\left( \rho^{2(i+k)} \sum_{\ell=1}^\infty \rho^{2\ell} \sum_{j=1}^\infty \rho^{4j}   
		\right) \\
		& = O( \rho^{2(i+k)}).
	\end{align*}
	Hence, as above, $ 	F_{n,ik}^{(1)}  = O( \rho^h ) $ if $ |i-k| \ge h $. The remaining terms can be bounded similarly. Indeed, examining $ F_{n,ik}^{(2)} $ under Assumption (W2) gives
	\[
	| F_{n,ik}^{(2)} | = O\left( \sum_{\ell=1}^\infty \sum_{j=0}^\infty \rho^{i+j+i+j+\ell+k+j+\ell+k+j} \right) = O( \rho^h ),
	\]
	and
	\[
	| F_{n,ik}^{(3)} | = O\left( \sum_{\ell=1}^\infty \sum_{j=0}^\infty\rho^{i+j+i+j+\ell+k+j+\ell+k+j} \right) = O( \rho^h ).
	\]
	Lastly,
	\[
	| F_{n,ik}^{(4)} | = O\left( \sum_{\ell=1}^\infty \sum_{j=0}^\infty\rho^{i+j+i+j+\ell+k+j+\ell+k+j} \right) = O( \rho^h ).
	\]
	This completes the proof.
\end{proof}

\begin{proof}[Proof of Lemma~\ref{lemma2}]  Denote by $ \hat{F}_{ni} $ and  $ \hat{F}_{n,ik} $ the quantities $ F_{ni} $ and $ F_{n,ik} $ when replacing $ \vecv_i $ by $ \vecv_{ni} $ and $ \vecw_{i} $ by $ \vecw_{ni} $. Denote by $  \hat{F}_{n,ik} = \sum_{l=1}^4 \hat{F}_{n,ik}^{(l)} $ the corresponding decomposition as in the previous proof. By assumption
	$ | v_{n\nu}^{(i)} - v_\nu^{(i)} | = O( r_n | v_\nu^{(i)} | ) $, where the $O$ does not depend on $ \nu $ and $ i $. Write
	\[
	v_{n\nu}^{(i)} w_{n\mu}^{(k)} - v_\nu^{(i)} w_\mu^{(k)} =|v_{n\nu}^{(i)} - v_\nu^{(i)}] [w_{\mu}^{(k)} + \frac{w_{n\mu}^{(k)} - w_\mu^{(k)}}{ w_\mu^{(k)} } w_{\mu}^{(k)} ] + v_\nu^{(i)} [ w_{n \mu}^{(k)} - w_\mu^{(k)}]
	\]
	so that
	\begin{align*}
		| v_{n\nu}^{(i)} w_{n\mu}^{(k)} - v_\nu^{(i)} w_\mu^{(k)} | &= O( r_n | v_\nu^{(i·)} w_\mu^{(k)} | + r_n^2 | v_\nu^{(i)} w_\mu^{(k)} |) + O( r_n | v_\nu^{(i)} | r_n |w_\mu^{(k)}| ) \\
		&= O( r_n |v_\nu^{(i)} w_\mu^{(k)} | ),
	\end{align*}
	for all $i,k \ge 1 $.
	Iterating this argument with $ | v_{n\nu}^{(i)} w_{n\mu}^{(i)}  - v_{\nu}^{(i)} w_{\mu}^{(i)}  | = O( r_n | v_{\nu}^{(i)} w_{\mu}^{(i)} | ) $ and
	$ |v_{n\nu'}^{(k)} w_{n\mu'}^{(k)} - v_{n\nu'}^{(k)} w_{n\mu'}^{(k)}| = O( r_n |v_{\nu'}^{(k)} w_{\mu'}^{(k)} | )$  yields
	\begin{equation}
		\label{iterated}
		| (v_{n\nu}^{(i)} w_{n\mu}^{(i)} ) (v_{n\nu'}^{(k)} w_{n\mu'}^{(k)})
		- (v_{\nu}^{(i)} w_{\mu}^{(i)} ) (v_{\nu'}^{(k)} w_{\mu'}^{(k)}) |
		= O( r_n | v_{\nu}^{(i)} w_{\mu}^{(i)} v_{\nu'}^{(k)} w_{\mu'}^{(k)} | ).
	\end{equation}
	This implies
	\[
	|\hat F_{ni} - F_{ni}| = O\left(  r_n \sum_{\nu, \mu=1}^{d_n} | w_\mu^{(k)} v_\nu^{(i)} | \left| \sum_{j=1}^\infty c_j^{(\nu)} c_j^{(\mu)} \right| \right) = O(r_n)
	\]
	and in view of (\ref{iterated}) we also get
	\[
	| \hat{F}_{n,ik}^{(l)} - F_{n,ik}^{(l)} | = O( r_n ), \qquad l = 1, \ldots, 4.
	\]
	The assertion now follows from the formulas for $ \beta_{ik}^{(n)} $ and $ \beta_{ik}^{(n)}$, since $ \hat\beta_{ik}^{(n)} = \beta_{ik}^{(n)}  + O(r_n) $. For $ i, k $ with $ |i-k| \ge h $ we have $ \hat\beta_{ik}^{(n)}  = O( \rho^h ) + o(1) $, where the $o(1) $ term is bounded by $ O( \rho^h ) $ if $n$ is large enough. Thus, $  \hat\beta_{ik}^{(n)}  = O( \rho^h ) $, which proves the second assertion.  
\end{proof}

\begin{proof}[Proof of Theorem~\ref{EVT}] We use the results of  \cite{Deo1972} and \cite{Deo1973} on Berman's condition for maxima of (absolute values of) non-stationary Gaussian processes: Let $ \{ \xi_t : t \in \N \} $ be a Gaussian process with $ E( \xi_t ) = 0 $ and $ E( \xi_t^2 ) = 1 $ for all $ t \in \N $. Assume that
	\[
	\rho_n = \sup_{|i-j| \ge n} | \Cov( \xi_i, \xi_j ) |
	\]
	satisfies
	\begin{itemize}
		\item[(I)] $\quad  \sum_{n=1}^\infty \rho_n^2 < \infty $, or
		\item[(II)] $ \quad \exists \beta > 0: \rho_n ( \log n )^{2+\beta} \to 0 $, as $ n \to \infty $ or
	\end{itemize}
	Then, by \cite[Lemma~1]{Deo1972}, for each $ z \in \R $ it holds
	\[
	P\biggl( a_n^{-1}\bigl( \max_{1 \le k \le n} | \xi_k | - b_n \bigr) \le z \biggr) \to \exp( - e^{-z} ),
	\]
	as $ n \to \infty $, where $ a_n = (2 \log n)^{-1/2} $ and $ b_n = (2 \log n)^{1/2} - (8 \log n)^{-1/2}( \log \log n + 4 \pi - 4) $. Further, see \cite[Theorem~2]{Deo1973}, if
	\begin{itemize}
		\item[(III)] $ \quad \rho_n = o(1) $, and
		\item[(IV)] $ \quad \sum_{1 \le i < j \le n} | \Cov( \xi_i, \xi_j ) | = O( n^{2-\gamma} ) $ for some $ \gamma > 0 $,
	\end{itemize}
	then
	\[
	P\biggl( a_n^{-1}\bigl( \max_{1 \le k \le n}  \xi_k  - c_n \bigr) \le z \biggr) \to \exp( - e^{-z} ),
	\]
	with 
	$ 	c_n =   (2 \log n )^{-1/2}- \frac{1}{2}(2 \log n)^{-1/2}( \log \log n + 4 \pi).
	$
	
	Let $ \calB = ( \calB_j )_j $ denote the Brownian motion with mean zero and covariances
	$ ( \beta_{ik} )_{i,k} $ and put $ \calB^* = ( \calB_j^* )_j = ( \calB_j / \beta_{jj}^{1/2} )_j $. By Lemma~\ref{lemma1} the above condition is satisfied, so that we can conclude that 
	\[
	P\left( a_n^{-1}\biggl( \max_{1 \le j \le n} | \calB_j^* | - b_n \biggr) \le z \right) \to \exp( -e^{-z} ), 
	\]
	as $ n \to \infty $, for all $ z \in \R $. 
	
	Put  $ A_m(z)= a_m z + b_m $ and
	$ \calD_n^* = ( \calD_{nj}^* )_j $ with $ \calD_{nj}^* = \calD_{nj} / \beta_{jj}^{1/2} $. From the results of \cite{StelandSachsSPA}, which considers a specific high-dimensional linear time series model,  \cite{BoursSteland2021}, where the case of a general multivariate linear process is treated, and \cite{MiesSteland2022}, which implies that almost the rate of approximation $ o_P( n^{-1/6} ) $ can be achieved under appropriate (moment) assumptions and treats nonlinear, non-stationary time series, there exists some $ 0 < \lambda < 1/6 $ such that, on a new probability space, 
	\[
	\max_{j \le m} | \calB^*_j - \calD_{nj}^* | = o_P( n^{-\lambda} ),
	\]
	a.s.  Then for $ z \ge 0 $
	\begin{align*}
		p_n(z) &= P\left( a_m^{-1} \biggl( \max_{1 \le j \le m} | \calD_{nj} |/ \beta_{jj}^{1/2} - b_m \biggr) \le z \right) \\
		& = P\left( \max_{1 \le j \le m} | \calD_{nj}^* | \le A_m(z) \right) \\
		&  = P \left( \max_{1 \le j \le m} | \calB_j^* | \le A_m(z) 
		+ \max_{1 \le j \le m} |  \calB_j^* | - \max_{1 \le j \le L_n} | \calD_{nj}^* | \right).
	\end{align*}
	A lower bound is given by
	\begin{align*}
		p_n(z) & \ge P\left( \max_{1 \le j \le m} | \calB_j^* | \le A_m(z) - \max_{1 \le j \le m} | \calB_j^* - \calD_{nj}^* | \right) \\
		& = P\left( a_m^{-1} \biggl( \max_{1 \le j \le m} |  \calB_j^* | - b_m \biggr) 
		\le z - a_m^{-1} \max_{1 \le j \le m} | \calB_j^* - \calD_{nj}^* | \right),
	\end{align*}
	and an upper bound by
	\begin{align*}
		p_n(z) & \le P\left( \max_{1 \le j \le m} | \calB_j^* | \le A_m(z) + \max_{1 \le j \le m} | \calB_j^* - \calD_{nj}^* | \right) \\
		& = P\left( a_m^{-1} \biggl( \max_{1 \le j \le m} |  \calB_j^* | - b_m \biggr) 
		\le z + a_m^{-1} \max_{1 \le j \le m} | \calB_j^* - \calD_{nj}^* | \right).
	\end{align*}
	We have
	\begin{align*}
		0 & \le a_m^{-1} \max_{1 \le j \le m} | \calB_j^* - \calD_{nj}^* |
		= o_P( m  \sqrt{2 \log m } n^{-\lambda} ) = o_P(1),
	\end{align*}
	a.s. Recall that  if a  sequence $ \{ G, G_n \} $ of distribution functions on the real line with $ G $ continuous satisfies $ G_n(z) \to G(z) $ for all $z$, then Polya's theorem implies the uniform convergence $ \sup_z | G_n(z) - G(z) | \to 0 $. Therefore, we obtain for $ z \ge 0 $
	\[ p_n(z) \to \lim_{n \to \infty} P\left( a_m^{-1}\left( \max_{1 \le j \le m} | \calB_j | - b_n \right) \le z \right) = \exp( -e^{-z} ).
	\]
	This completes the proof.
\end{proof}

\end{document}